\documentclass[a4paper, 11pt]{article}

\usepackage[hang,small]{caption}
\usepackage[margin=3cm]{geometry}
\usepackage[latin1]{inputenc}
\usepackage{amsmath}
\usepackage{amsfonts}
\usepackage{amssymb}
\usepackage{amsthm}
\usepackage{bm}						
\usepackage{mdwlist}					
\usepackage[pdftex]{graphicx}
\usepackage{fancyhdr}
\usepackage{float}
\usepackage{cite}
\usepackage{enumerate}          
\usepackage{fancybox}           

\usepackage{url}
\title{Large Deviations for Non-Crossing Partitions}
\author{Janosch Ortmann\\ Warwick Mathematics Institute}

\numberwithin{equation}{section}

\theoremstyle{definition}

\newtheorem{defn}[equation]{Definition}
\newtheorem{ex}[equation]{Example}

\newtheorem{rmk}[equation]{Remark}

\theoremstyle{plain}
\newtheorem{thm}[equation]{Theorem}
\newtheorem{prop}[equation]{Proposition}
\newtheorem{lem}[equation]{Lemma}
\newtheorem{cor}[equation]{Corollary}

\renewcommand{\c}{\ensuremath{\mathcal{C}}}
\newcommand{\C}{\ensuremath{\mathbb{C}}}
\renewcommand{\d}{\mathrm d}

\newcommand{\E}{\ensuremath{\mathbb{E}}}
\newcommand{\e}{\ensuremath{\mathcal{E}}}

\newcommand{\g}{\ensuremath{\mathfrak{G}}}

\renewcommand{\i}{\ensuremath{\mathcal{I}}}

\newcommand{\Lb}{\ensuremath{\bm{L}}}

\newcommand{\map}{\ensuremath{\longrightarrow}}
\newcommand{\Map}{\ensuremath{\longmapsto}}

\newcommand{\m}{\ensuremath{\mathfrak{M}}}
\newcommand{\N}{\ensuremath{\mathbb{N}}}
\newcommand{\NC}{\ensuremath{\text{NC}}}
\newcommand{\one}{\ensuremath{\mathbf{1}}}
\newcommand{\p}{\ensuremath{\mathcal{P}}}
\renewcommand{\P}{\ensuremath{\mathbb{P}}}
\newcommand{\Q}{\ensuremath{\mathbb{Q}}}
\newcommand{\R}{\ensuremath{\mathbb{R}}}
\newcommand{\s}{\ensuremath{\sigma}}
\newcommand{\Sb}{\ensuremath{\bm{S}}}

\newcommand{\ul}[1]{\underline{#1}}
\newcommand{\w}{\ensuremath{\mathcal{W}}}
\newcommand{\wh}[1]{\widehat{#1}}
\newcommand{\wt}[1]{\widetilde{#1}}
\newcommand{\x}{\ensuremath{\mathcal{X}}}
\newcommand{\y}{\ensuremath{\mathcal{Y}}}
\newcommand{\Z}{\ensuremath{\mathbb{Z}}}

\newcommand{\abs}[1]{\ensuremath{\left|#1\right|}}

\newcommand{\inv}[1]{\ensuremath{\frac{1}{#1}}}
\newcommand{\norm}[1]{\ensuremath{\|#1\|}}

\newcommand{\spt}{\mathrm{supp}}
\newcommand{\uind}[2]{\ensuremath{#1^{(#2)}}}

\newcommand{\ts}[1]{\textsc{#1}}

\DeclareMathOperator*{\diag}{diag}

\begin{document}

\maketitle

\abstract{\noindent We prove a large deviations principle for the empirical law of the block sizes of a uniformly distributed non-crossing partition. As an application we obtain a variational formula for the maximum of the support of a compactly supported probability measure in terms of its free cumulants, provided these are all non-negative. This is useful in free probability theory, where sometimes the R-transform is known but cannot be inverted explicitly to yield the density.}

\fancyhead[R]{\nouppercase{\leftmark}}
\fancyhead[L]{}
\cfoot{\thepage}

\bibliographystyle{acm}

\bibstyle{plain}

%
%

\section{Introduction}
\label{Sec:Introduction}

In this paper we study the block structure of a \emph{non-crossing partition} chosen uniformly at random. Any partition $\pi$ of the set $\ul{n}=\{1,\ldots,n\}$ can be represented on the circle by marking the points $1,\ldots,n$ and connecting by a straight line any two points whose labels are in the same block of the partition. The partition is then said to be \emph{non-crossing} if none of the lines intersect. These objects were introduced by G.~\ts{Kreweras} \cite{Kreweras72} and have been studied in the combinatorics literature as an example of a \emph{Catalan structure}.

We study the empirical measure defined by the blocks of a uniformly random non-crossing partition $\pi$. That is, if $\pi$ has $r$ blocks of sizes  $B_1,\ldots,B_r$ we consider the random probability measure on $\N$ defined by
\begin{align*}
	\lambda_n & = \inv{r} \sum_{j=1}^r \delta_{B_j}.
\end{align*}
We will prove that the sequence $\big(\lambda_n\big)_{n\in\N}$	satisfies a large deviations principle of speed $n$ on the space $\m_1(\N)$ of probability measures on the natural numbers.

This result is obtained via a construction of a uniformly random non-crossing partition by suitably conditioned independent geometric random variables. As a stepping-stone we establish a joint large deviations principle for the process versions of empirical mean and measure of that independent sequence.

A main application of the large deviations result comes from free probability theory. Often one can obtain the free cumulants of a non-commutative random variable. These cumulants characterise the underlying distribution, but obtaining the density involves locally inverting an analytic function which may not lead to a closed-form expression. In such a situation one would still hope to deduce some properties of the underlying probability measure, for example about its support.

The free analogue of the moment-cumulant formula expresses the moments of a non-commutative random variable in terms of its free cumulants. More precisely the moments can be written as the expectation of an exponential functional (defined in terms of the free cumulants) of a non-crossing partition, chosen uniformly at random. Knowing the large deviations behaviour of the latter allows us to apply Varadhan's lemma to describe the asymptotic behaviour of the moments. This in turn yields the maximum of the support of the underlying distribution in terms of the free cumulants.

\subsection*{Statement of Results}

Our first main result is the large deviations property of the random measures $(\lambda_n)_{n\in\N}$. In the form we are stating it here it is a direct corollary of Theorem~\ref{Thm:LDP}.

\begin{thm}
	\label{Thm:LDPLambdaOnly}
	The sequence $(\lambda_n)_{n\in\N}$ satisfies a large deviations principle in $\m_1(\N)$ with good convex rate function $J$ given by
\begin{align}
	\label{Eq:LDPLambdaOnly}
	J(\mu) & = \log 4  -\inv{m_1(\mu)} H(\mu) - \inv{m_1(\mu)} \log\big(m_1(\mu)-1\big) + \log\left( 1-\inv{m_1(\mu)} \right)
\end{align}
where $H(\mu)$ denotes the entropy of a probability measure $\mu$ and $m_1(\mu)$ its mean.
\end{thm}

\par\noindent Since $J(\nu)=0$ if and only if $\nu$ is the geometric distribution $\g_2$ of parameter $\inv{2}$, we obtain a law of large numbers as an immediate corollary. Namely $\lambda_n\longrightarrow \g_2$ almost surely in the weak topology.

In the proof of the theorem we need to work with the the function $m_1(\mu)$, which is not continuous in the weak topology. As a stepping-stone we therefore establish a joint large deviations principle for the path versions of empirical mean and measure of i.i.d. geometric random variables.

Theorem~\ref{Thm:LDPLambdaOnly} has an application in free probability. Namely it allows us to express the maximum of the support of a compactly supported probability measure in terms of its free cumulants, provided the latter are non-negative. For some background on free probability see Section~\ref{Sec:SpectralEdgeFormula} and the references given there.

\paragraph{Theorem~\ref{Thm:EdgeNonnegative}}
	Let $\mu$ be a compactly supported probability measure whose free cumulants $(k_n)_{n\in\N}$ are all non-negative. Then the maximum of the support $\rho_\mu$ of $\mu$ is given by
	\begin{align*}
		\log\left(\rho_\mu \right) & = \sup\left\{\inv{m_1(p)}\sum_{n\in L} p_n\log\left(\frac{k_n}{p_n}\right) - \frac{\Theta(m_1(p))}{m_1(p)} \colon p\in\m_1^1(L)\right\}
	\end{align*}	
	where $L$ is the set of $n\in\N$ such that $k_n\neq 0$ and $\m_1^1(L)$ denotes the set of $p\in\m_1^1(\N)$ with $p(L^c)=0$.

\paragraph{Acknowledgements.} The author would like to thank his PhD advisor, Neil O'Connell for his advice and support in the preparation of this paper. We also thank Philippe Biane, Jon Warren and Nikolaos Zygouras for helpful discussions and suggestions.


\section{Uniformly Random Non-Crossing Partitions}
\label{Sec:UnifNCP}

In this section we introduce the combinatorial objects mentioned in the introduction. We describe how to generate the uniform distribution on the set of Dyck paths or non-crossing partitions using two sequences of independent and identically distributed geometric random variables. This construction will be used in Section~\ref{Sec:LDPNCP} to prove the large deviations result, Theorem~\ref{Thm:LDP}.

\subsection{Catalan Structures}

A \emph{Dyck path} of \emph{semilength} $n$ is a lattice path in $\Z^2$ that never falls below the horizontal axis, starting at $(0,0)$ and ending at $(2n,0)$, consisting of steps $(1,1)$ (\emph{upsteps}) and $(-1,1)$ (\emph{downsteps}). Every such path consists of exactly $n$ up- and downsteps each. The set of Dyck paths of semilength $n$ is denoted by $\p(n)$. A maximal sequence of upsteps is called an \emph{ascent}, while a maximal sequence of downsteps is referred to as a \emph{descent}.

Non-crossing partitions were introduced by G.~\ts{Kreweras} \cite{Kreweras72}. A partition $\pi$ of the set $\ul{n}=\{1,\ldots,n\}$ is said to be \emph{crossing} if there exist distinct blocks $V_1$, $V_2$ of $\pi$ and $x_j,y_j\in V_j$ such that $x_1<x_2<y_1<y_2$. Otherwise $\pi$ is said to be \emph{non-crossing}. Equivalently label the vertices of a regular $n$-gon $1,\ldots,n$, then $\pi$ is non-crossing if and only if the convex hulls of the blocks are pairwise disjoint. 

\begin{figure}[H]
	\begin{center}
	\includegraphics[scale = .2]{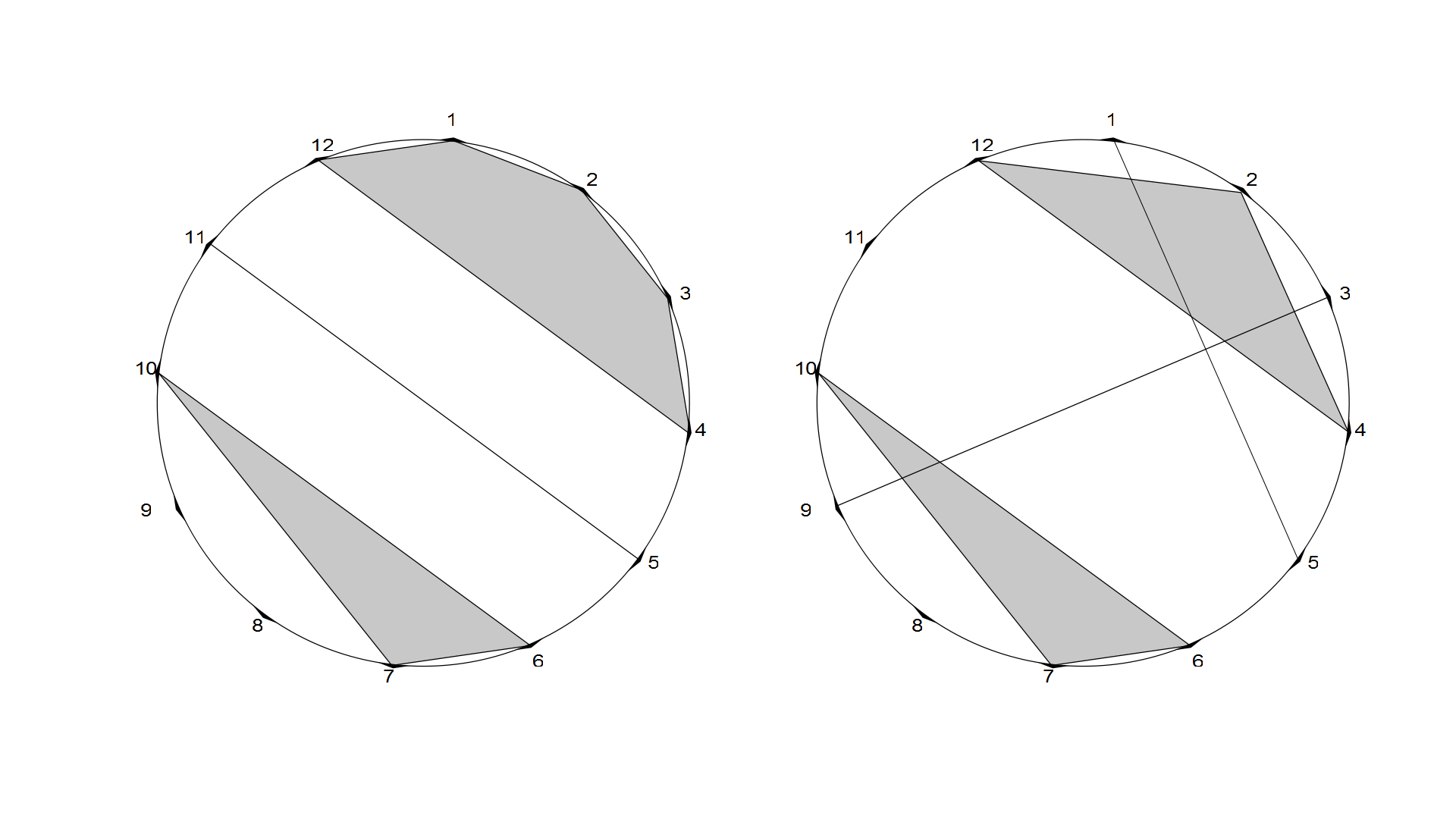}
	\hspace{5cm}
	\parbox{11cm}{\caption*{\footnotesize{The partition $\{\{8\}, \{9\}, \{10,7,6\}, \{11,5\}, \{12,4,3,2,1\} \}$ is non-crossing, $\{\{5,1\}, \{8\}, \{9,3\},\{10,7,6\}, \{12,4,2\} \}$ is crossing.}}}
\end{center}

\end{figure}

\par\noindent The set of all non-crossing partitions of $\ul{n}$ is denoted by $\NC(n)$. Combinatorial results on non-crossing partitions, including instances where they arise in topology and mathematical biology can be found in R.~ \ts{Simion}'s survey \cite{Simion00}. For other areas of mathematics where non-crossing partitions appear see also \ts{McCammond} \cite{McCammond06}. The role of non-crossing partitions in free probability is detailed in Section~\ref{Sec:SpectralEdgeFormula}.

There exists a well-known bijection $\Phi\colon \p(n)\map\NC(n)$ which maps the descents of $p\in\p_n$ to the blocks of $\Phi(p)$, see for example \ts{Callan} \cite{Callan08} or \ts{Yano--Yoshida} \cite{YanoYoshida07}. Given $p\in\p_n$ label the upsteps from left to right by $1,\ldots,n$. Label each downstep by the same index as its corresponding upstep, that is the first upstep to the left on the same horizontal level. Then the descents induce an equivalence relation on $\ul{n}$: two labels are equivalent if and only if the corresponding downsteps are part of the same descent. The associated partition is then easily seen to be non-crossing. 

Conversely, given $\pi=\{V_1,\ldots,V_r\}\in\NC(n)$ write the elements of each block $V_j$ in descending order, then sort the blocks in ascending order by their least elements. This gives the descent structure of $\Phi^{-1}(\pi)$, which can be completed by the ascents in a unique way to form a Dyck path.

\begin{figure}[H]
	\begin{center}
	\includegraphics[scale = .2]{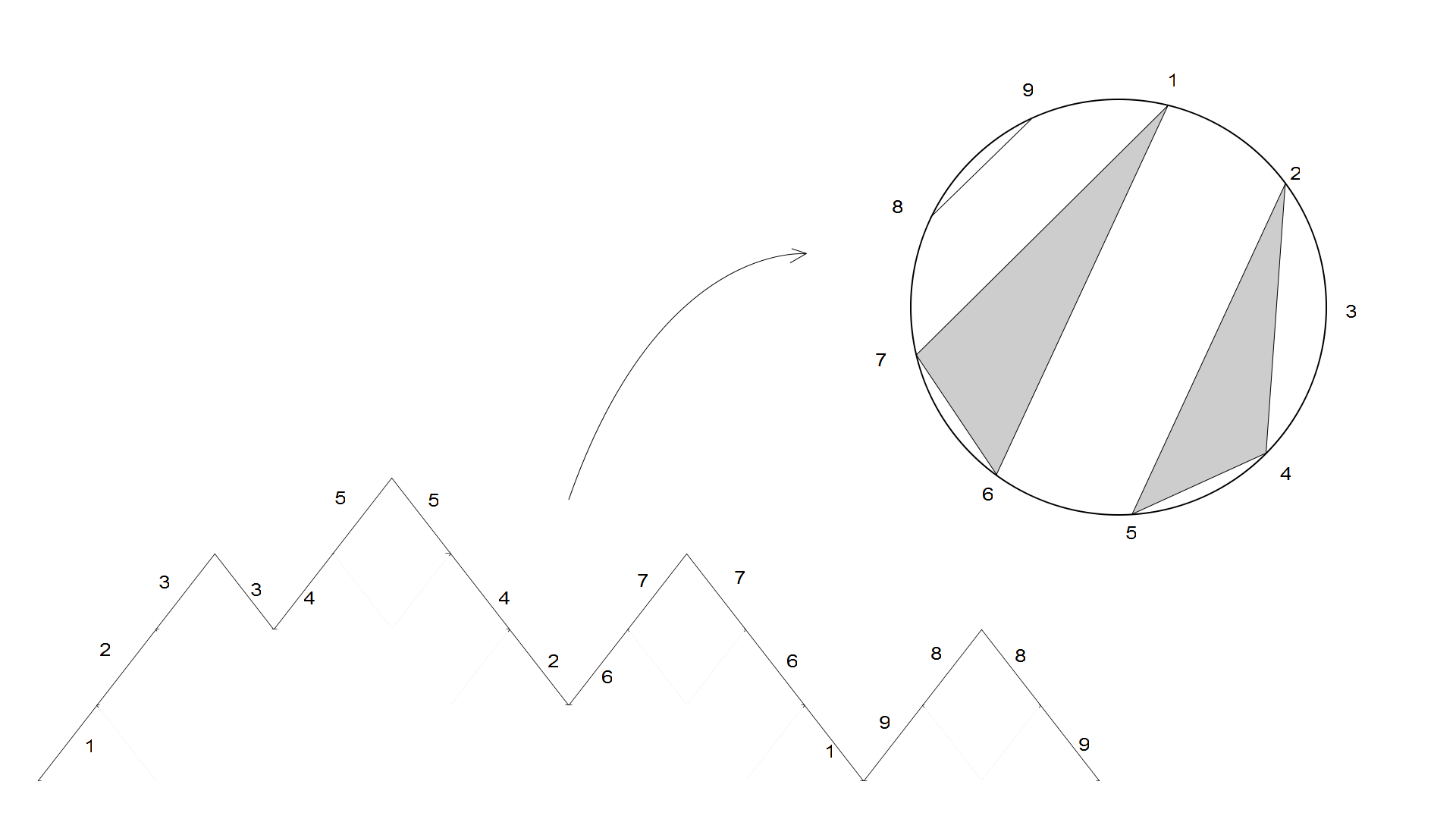}
	\hspace{10cm}
	\parbox{11cm}{\caption*{\footnotesize{An example for the bijection $\Phi$.}}}
\end{center}

\end{figure}

The common cardinality of $\p(n)$ and $\NC(n)$ is $C_n=\frac{(2n)!}{n!(n+1)!}$, the $n^\text{th}$ \emph{Catalan number}. Such combinatorial objects are referred to as \emph{Catalan structures} and have been much studied. A list of Catalan structures has been compiled by R.~\ts{Stanley} \cite{Stanley2}, where many results and references on Catalan structures can also be found.

Of course our results extend to any statistic $s$ of any other Catalan structure $\s$ for which there exists a bijection $\Psi\colon \s\map \NC(n)$ that maps $s$ to the blocks of the image partition. Examples include the blocks of \emph{non-nesting partitions} (see \ts{Reiner} \cite{Reiner97}, Remark 2) or the length of chains in ordered trees as described in \ts{Prodinger} \cite{Prodinger83}.

\subsection{A Representation for the Uniform Measure}
\label{Sec:RepUnifM}

Since the sets $\NC(n)$ and $\p(n)$ are finite there exists a uniform distribution on them. Let $w$ have this distribution on $\p(n)$. Such a random variable is also referred to as a \emph{Bernoulli excursion}. We will study the descent structure of a $w$. Because of the above bijection this is equivalent to studying the blocks of a uniformly random element of $\NC(n)$. 


The number of Dyck paths with semilength $n$ and $k$ descents is given \cite{Deutsch99} by the \emph{Narayana numbers}

\begin{align*}
	N(n,k) = \inv{n}\begin{pmatrix}n\\ k \end{pmatrix} \begin{pmatrix} n \\ k-1\end{pmatrix}.
\end{align*}
Therefore the expected number of descents in $w$ is $\frac{n+1}{2}$. Further it follows from results in \cite{YanoYoshida07} (p.3153) that the expected number of descents of length 1 is given by $\frac{n^2+n}{4n-4}$. Asymptotically we therefore have about $\frac{n}{2}$ descents, roughly half of which are singletons.  
However there do not seem to be asymptotic results beyond the singleton descents
in the literature.  Heuristic arguments suggest that about half of the remaining descents is of length 2 and so on, and indeed, the law of large numbers mentioned in Section~\ref{Sec:Introduction} (Corollary~\ref{Thm:LLN}) confirms that this is the case.\\

\par\noindent We now construct a Bernoulli excursion using conditioned geometric random variables. For any $n\in\N$ let $b_n\colon \N^{2n}\map \bigcup_{k\in\N} \wh\p(k)$ (where $\wh\p(k)$ is the set of \emph{all} length $2k$ lattice paths on \Z, starting at zero and consisting of steps $(1,1)$ and $(1,-1)$)  denote the map that reconstructs a path from a sequence of ascents and descents. That is, $b_n(x_1,y_1,\ldots,x_n,y_n)$ is the path described by $x_1$ upsteps, $y_1$ downsteps, then $x_2$ upsteps and so on, terminating with $y_n$ downsteps.

Let $X_n,Y_n$ be i.i.d. random variables with common law given by the geometric distribution with parameter $\inv{2}$. We will view these as the subsequent ascents and descents of a simple random walk $\Sigma$ on \R\ starting at 0 with an upstep. Denote by
\begin{align*}
	T_n&:= \sum_{j=1}^n(X_j+Y_j)
	\intertext{the combined length the first $n$ up- and downsteps take in total and let $\wh\tau_n$ be the number of descents completed after $2n$ steps of the simple random walker:}
	\wh\tau_n & = \max\{k\in\N\colon T_k\leq 2n\}.
\end{align*}
We will later work with a renormalisation of $\wh{\tau}_n$, namely $\tau_n = \frac{\wh\tau_n}{2n}$. We denote by $E_n$ the event that $b_{\tau_n}\left(X_1,Y_1, \ldots, X_{\tau_n},Y_{\tau_n}\right)$ is a Dyck path of semilength $n$:
\begin{align}
	\label{Eq:DefEN}
	E_n & = \left\{T_{\wh\tau_n}=2n,\, \sum_{j=1}^{\wh{\tau}_n}X_j = \sum_{j=1}^{\wh{\tau}_n}Y_j,\, \sum_{j=1}^r\left(X_j-Y_j\right)\geq 0\ \forall\, j<\wh\tau_n \right\}.
\end{align}
The following lemma is now straightforward to check.

\begin{lem}
Conditioned on $E_n$ the distribution of $w_{\wh\tau_n}\left(X_1,Y_1, \ldots, X_{\wh\tau_n},Y_{\wh\tau_n}\right)$ on $\p(n)$ is uniform. Hence, conditioned on $E_n$, the random measure $\lambda_n$ defined by
\begin{align}
\label{Eq:DefLambdaN}
\lambda_n=\inv{\wh\tau_n} \sum_{j=1}^{\wh\tau_n}\delta_{Y_j}
\end{align}
is the empirical measure of the descents of a Bernoulli excursion or, equivalently, the block sizes of a uniformly random element of $\NC(n)$.
\end{lem}


\section{Process Level Large Deviations}

Let $(X_n)_{n\in\N}$ be an i.i.d. sequence of geometric random variables with parameter $\inv{2}$ and denote their common law by $\g_2$. We define processes  $\Sb_n$, $\Lb_n$, indexed by the unit interval and taking values in the space of real numbers and positive finite measures on \N\ respectively by
\begin{align}
	\label{Eq:DefSPath}
	\Sb_n(t) & = \inv{n}\sum_{j=1}^{\lfloor nt\rfloor} X_j+\left(s-\frac{\lfloor ns\rfloor}{n}\right) X_{\lfloor ns\rfloor + 1} \\ 
	\label{Eq:DefLPath}
	\Lb_n(t) & = \inv{n}\sum_{j=1}^{\lfloor nt\rfloor} \delta_{X_j} +\left(s-\frac{\lfloor ns\rfloor}{n}\right) \delta_{X_{\lfloor ns\rfloor + 1}}.
\end{align}
In this section we prove a large deviations principle for the pair $(\Sb_n,\Lb_n)$. We start by proving a joint LDP for the pair of end-points $\left(\Sb_n(1),\Lb_n(1)\right)$ via a projective limit argument. We then adapt arguments from \ts{Dembo--Zajic} \cite{DemboZajic95} to obtain the path-wise result.

\begin{rmk}
The reason for obtaining this joint large deviations principle is that for our main large deviations result we need to use the mean as well as the empirical measure but the map $\mu\Map m_1(\mu)$ is not continuous in the weak topology. An alternative would have been a priori to strengthen the topology on $\m_+(\N)$ to the \emph{Monge--Kantorovich topology}, the coarsest topology that makes the map $m_1$ continuous and is finer than the weak topology. However results by \ts{Schied} \cite{Schied98} show that in this topology Sanov's theorem does not hold for geometric random variables, because this distribution does not possess all exponential moments.
\end{rmk}

\par\noindent Let us first recall the definition of a large deviations principle. For background on large deviations theory see for example \ts{Dembo--Zeitouni} \cite{DemboZeitouni}.

\begin{defn}
A sequence of measures $(\mu_n)_{n\in\N}$ taking values on a Polish space is said to satisfy a \emph{large deviations principle} of \emph{speed} $a=(a_n)_{n\in\N}$ with \emph{rate function} $I$ if $a$ is a strictly increasing sequence diverging to $\infty$, $I$ is lower semi-continuous, has compact level sets and
	\begin{align}
		\label{Def:DefLowerBound}
		\liminf_{n\to\infty} \inv{a_n}\log\mu_n(G) & \geq -\inf_{x\in G} I(x)\\
		\label{Def:DefUpperBound}
		\limsup_{n\to\infty} \inv{a_n}\log\mu_n(F) & \leq -\inf_{x\in F} I(x)
	\end{align}
	for every open set $G$ and every closed set $F$. (\ref{Def:DefLowerBound}) and (\ref{Def:DefUpperBound}) are often referred to as the large deviations \textit{lower bound} and \textit{upper bound} respectively.
\end{defn}

\par\noindent All the large deviations principles considered in this paper will be of speed $n$, that is $a_n=n$ for all $n\in\N$.

\subsection{Joint Sanov and Cram\'er}
\label{Sec:JointSanovCramer}

\par\noindent Denote by $\m_+(\N)$ the space of finite measures on \N\ and let $\m_1(\N)$ be the subset of probability measures. We equip $\m_(\N)$ with the topology of weak convergence. This topology is induced by the complete separable metric $\beta$ given for $\mu,\nu\in\m_+(\N)$ by
\begin{align}
	\label{Eq:DefBeta}
	\beta(\mu,\nu) & = \sup\left\{\abs{\int f\d\mu - \int f\d\nu }\colon \|f\|_L+\|f\|_\infty\leq 1\right\}
\end{align}
where $\|\cdot\|_L$ denotes the Lipschitz norm. So $\m_+(\N)$ is a Polish space, and so is $\m_1(\N)$ when equipped with the subspace topology. See Appendix A of \ts{Dembo--Zajic} \cite{DemboZajic95}.

Our goal here is to establish a joint large deviations principle on $\y:= \R\times\m_+(\N)$ for the empirical mean and measure of the $X_n$. To be more precise define random elements $S_n:= \Sb_n(1)\in\R$ and $L_n:= \Lb_n(1)\in\m_1(\N)$. By Cram\'er's theorem and Sanov's theorem respectively, the laws of $S_n$, $L_n$ already satisfy a large deviations principle on \R\ and $\m_1(\N)$ individually. The point here is to show that this also holds for the pair. Recall that $m_1(\mu)$ denotes the mean of a probability measure $\mu$.

\begin{prop}
\label{Thm:JointSanovCramer}
Let $\eta_n$ denote the law of $(S_n,L_n)$. Then $\left(\eta_n \right)_{n\in\N}$ satisfies a large deviations principle in $\y$ with good rate function $I_1$ given by
\begin{align}
	\label{Eq:IOne}
I_1(x,p)& =\begin{cases}
	H(p|\g_2) \quad & \text{if } p\in\m_1(\N)\text{ and }m_1(p)=x\\
	+\infty & \text{otherwise}
	\end{cases}
\intertext{where $H(\cdot\vert \cdot)$ denotes the \emph{relative entropy} of two probability measures, i.e.,}
\notag
H(\nu\vert \g_2) &= \sum_{m=1}^\infty \nu_m\log\left(\frac{\nu_m}{2^{-m}}\right) = m_1(p)\log(2)-H(\nu)
\end{align}
and $H(p)=-\sum_m p_m\log(p_m)$ is the \emph{entropy} of a probability measure $p$.

\begin{proof}
Recall that the weak topology on $\m_1(\N)$ is induced by the dual action of the space $\c_b(\N)$ of bounded continuous functions on \N. Fix a finite collection $f_1,\ldots,f_d\in\c_b(\N)$. The random variables $\big(X_n,f_1(X_n),\ldots,f_d(X_n)\big)\in\R^{d+1}$ are independent and identically distributed, so by Cram\'er's theorem \cite[Corollary 6.1.6]{DemboZeitouni} their laws satisfy a large deviations principle on $\R^{d+1}$ with good convex rate function given by
\begin{align*}
	\Lambda_{f_1,\ldots,f_d}^*(x_1,\ldots,x_{d+1}) & = \sup\left\{\sum_{j=1}^{d+1} \lambda_jx_j - \log\E e^{\lambda_1X_1+\sum_{j=1}^d\lambda_{j+1}f_j(X_1)} \colon\lambda\in \R^{d+1} \right\}
\end{align*}
The idea is now to take a projective limit approach, following closely Section 4.6 in \cite{DemboZeitouni}. We first construct a suitable projective limit space in which $\R\times\m_+(\N)$ can be embedded. The proposition will follow from an application of the Dawson--G\"artner theorem.

Denote $\w=\c_b(\N)$ and let $\w'$ be its algebraic dual, equipped with the $\tau(\w',\w)$-topology, that is the weakest topology making the maps $\w'\ni f\Map f(w)\in\R$ continuous for all $w\in\w$. Let further $J$ be the set of finite subspaces of $\w$, partially ordered by inclusion. For each $V\in J$ define $\y_V=\R\times V'$ and equip it with the $\tau(\R\times V',\R\times V)$-topology. Defining now projection maps $p_{U,V}$ for each $U\subseteq V$ by
\begin{align*}
	p_{U,V} &\colon \y_V\map \y_U\\
		p_{U,V}(x,f) & = \left(x,f\vert_U\right),
\intertext{we obtain a projective system $\left(\y_V,p_{U,V}\colon U\subseteq V\in J\right)$. Denote by $\wt{\x}$ its projective limit, equipped with the subspace topology from the product topology. Let further $\x=\R\times\w'$ and define $\Phi\colon\x\map\wt{\x}$ by}
	\Phi(x,f)& = \left((x,f\vert_V)\colon V\in J\right).
\end{align*}
This is clearly a bijection. Using the definition of the weak topology in terms of open balls, as in Chapter 8 of \ts{Bollob\'as} \cite{Bollobas}, it is clear that $\Phi$ is actually a homeomorphism.

Next we embed $\R\times\m_1(\N)$ into $\x$: for $(x,\mu)\in\R\times\m_1(\N)$ let
\begin{align*} 
	\Psi(x,\mu)=(x,[h\Map\int h\d \mu])\in\x.
\end{align*} 
Then $\Psi$ is a homeomorphism onto its image, which we denote by \e. Let $\wt{\eta}_n$ be the image measure of $\eta_n$ under $\Psi$. By the Dawson--G\"artner theorem and the finite-dimensional large deviations principle mentioned above, these satisfy a large deviations principle with good rate function $I_\Psi$ given by
\begin{align*}
	I_\Psi(x,f) & = \sup \left\{ \Lambda^*_{\lambda_1,\ldots\lambda_d} \left(x, f(\lambda_1), \ldots, f(\lambda_d)\right)\colon \lambda_1,\ldots \lambda_d\in\w\ \right\}.
\end{align*}
By Cram\'er's and Sanov's theorem respectively we have exponential tightness for the sequences $\left(S_n\right)_{n\in\N}$ and $\left(L_n\right)_{n\in\N}$ separately. Therefore the sequence of \emph{pairs} $\left(\left(S_n,L_n\right)\right)_{n\in\N}$ is exponentially tight in $\R\times\m_1(\N)$. The inverse contraction principle (see Theorem 4.2.4 in \cite{DemboZeitouni} and the remark (a) following it) now yields the desired LDP for $(S_n,L_n)$ with the good rate function
\begin{align*}
I_1(x,\mu)& =\sup_{f_1,\ldots,f_d\in\c_b(\N)} \Lambda_{\lambda_1,\ldots,\lambda_d}\left(x,\int f_1\,\d\mu, \ldots, \int f_d\,\d\mu \right).
\end{align*}
It remains to show that $I_1$ is actually of the form (\ref{Eq:IOne}). Suppose that $m_1(\mu)=x$, let $(\lambda_1,\ldots, \lambda_{d+1})\in\R^{d+1}$ and define $\phi(y)=\lambda_1y + \sum_{j=1}^d \lambda_{j+1} f_j(y)$. By Jensen's inequality,  
\begin{align*}
	\log\E\left[e^{\phi(Y)} \right] & \geq \int\phi \, \d\mu -H(\mu\vert\g_2) = \lambda_1 x + \sum_{j=1}^d\lambda_{j+1}\int f_j \, \d\mu - H(\mu\vert\g_2).
\end{align*}
So $\Lambda^*_{f_1,\ldots,f_d}\left(x,\int f_1\,\d\mu, \ldots, f_1\,\d\mu\right) \leq H(\mu\vert\g_2)$ and therefore $	I_1(x,\mu) \leq H(\mu\vert\g_2)$.

If $\mu$ is a Dirac mass then $H(\mu\vert\g_2)=0$ and the inequality $I_1(x,\mu)\geq H(\mu\vert\g_2)$ follows trivially. So we assume that $\mu$ is not a Dirac mass. Define $e_j\in\c_b(\N)$ by $e_j(m)=\delta_{jm}$. Write $\spt\left(\mu\right) = \left\{n_k\colon k\in J\right\}$. Then,
\begin{align*}
	I_1(x,\mu) &\geq \sup\left\{\Lambda^*_{e_{n_1},\ldots,e_{n_d}}\left( x,\mu_{n_1}, \ldots, \mu_{n_d}\right) \colon d\in J \right\}\\
	& = \sup_{\lambda\in (-\infty,\log(2))\times\R^{d}}\left\{\lambda_1x+\sum_{j=1}^d \lambda_{j+1} \mu_{n_j} - \log\E\left[e^{\lambda_1X_1+\sum_{j=1}^d \lambda_{j+1}e_{n_j}(X_1)}\right] \colon  \right\}.
\end{align*}
Fix $d\in J$, and let $g(\lambda)$ denote the function inside the supremum. The effective domain of $\Lambda_{e_{n_1},\ldots,e_{n_d}}$ is $(-\infty,\log(2))\times \R^d$. Because $\mu$ is not a Dirac mass the function $g(\lambda)$ tends to $-\infty$ whenever $\abs{\lambda}$ tends to $ \infty$, in whatever direction. So the supremum of $g$ is attained at some $\lambda_0 \in (-\infty,\log(2))\times \R^d$. Then $\lambda_0$ is a local maximum for $g$, whence $\nabla g(\lambda_0)=0$, or equivalently $\nabla\Lambda_{e_{n_1}, \ldots,e_{n_d}}(\lambda_0) = \left( x,\mu_{n_1} ,\ldots, \mu_{n_d}\right)^T$. So we can define an exponential tilting $\nu_{\lambda_0}$ of $\mu$ by
\begin{align*}
	\nu_{\lambda_0}(\d y) & = e^{\lambda_1y+\sum_{j=1}^d \lambda_{j+1} e_{n_j}(y) - \Lambda_{e_{n_1}, \ldots,e_{n_d}}(\lambda_0)}\,\mu(\d y).
\end{align*}
The probability measure $\nu_{\lambda_0}$ has mean $x$ and satisfies $\int e_{n_j} \,\d\mu=\int e_{n_j}\,\d\nu_{\lambda_0}$ for all $j\in\{1,\ldots,d\}$. Moreover,
\begin{align*}
	H\left(\nu_{\lambda_0}\vert\g_2\right) & = \lambda_1x + \sum_{j=1}^d \lambda_{j+1} \mu_{n_j} - \Lambda_{e_{n_1}, \ldots, e_{n_d}}(\lambda_0) \leq \Lambda^*_{e_{n_1}, \ldots, e_{n_d}} \left( x,\mu_{n_1}, \ldots, \mu_{n_d}\right)
\intertext{and therefore,}
I_1(x,\mu) & \geq \sup_{d\in J} \inf\left\{H\left(\nu\vert\g_2\right)\colon m_1(\nu)=x ,\, \nu_{n_j}=\mu_{n_j} \forall j\in\{1,\ldots,d\} \right\}\\
	& = H\left(\mu\vert \g_2\right).
\end{align*}
The same estimate shows that $I_1(x,\mu)=+\infty$ whenever $m_1(\mu) \neq x$.

It now follows from Lemma 4.1.5 (a) in \cite{DemboZeitouni} that the LDP also holds in the larger space $\y=\R\times\m_+(\N)$, by setting $I_1(x,\mu)=\infty$ whenever $\mu$ is not a probability measure.
\end{proof}
\end{prop}

\subsection{The Sample-Path Result}
\label{Sec:SamplePathResult}

\begin{thm}
	\label{Thm:SamplePathJSC}
	Let $\xi_n$ denote the law of $(\Sb_n,\Lb_n)$ on $\c\left([0,1];\y\right)$, the space of continuous functions from the unit interval to \y. The sequence $\left(\xi_n \right)_{n\in\N}$ satisfies a large deviations principle on $\c\left([0,1];\y\right)$ with good rate function $I_2$ given by
	\begin{align}
		\label{Eq:DefITwo}
		I_2(\bm{x},\bm{p}) & = \begin{cases}
											\int_0^1 H \left( \dot{\bm{p}} (s) \vert \g_2 \right)\,\d s \quad & \text{if } (\bm{p},\bm{x})\in\e\\
						 +\infty & \text{otherwise}
										\end{cases}
	\end{align}
	where $\e$ is the space of elements $(\bm{m},\bm{p})$ of absolutely continuous maps $[0,1]\map\y$ such that $(\bm{m}(0),\bm{p}(0))=0$, the map $s\Map \bm{m}(s)$ is differentiable almost everywhere, $\bm{p}(t)-\bm{p}(s)\in\m_{t-s}(\N)$ and the limit
\begin{align*}
	\dot{\bm{p}}_t & = \lim_{\epsilon\to 0} \frac{\bm{p}_{t + \epsilon} - \bm{p}_t}{\epsilon}
\end{align*}
exists in the weak topology for almost every $t\in [0,1]$ and has the property that $m_1(\dot{\bm{p}}(\cdot))=\bm{m}(\cdot)$.
\end{thm}

\par\noindent For $\left(\Lb_n(\cdot)\right)$ on its own the analogous result can be found in \ts{Dembo--Zajic} \cite{DemboZajic95} and we will use a similar structure, using the joint large deviations principle for empirical mean and measure established above. We first prove exponential tightness for the pair of paths:

\begin{lem}
	\label{Thm:SLExpTight}
	$\left(\left(\Sb_n(\cdot), \Lb_n(\cdot)\right) \right)_{n\in\N}$ is exponentially tight in $\c\left([0,1];\y \right)$.
	\begin{proof}
		Let the distance $d$ on \y\ be given by
	\begin{align*}
		d\left(\left(x_1,\mu_1\right), \left(x_2,\mu_2\right) \right) &= \abs{x_1-x_2} + \beta(\mu_1,\mu_2).
\end{align*}
By Lemma~A.2 in \cite{DemboZajic95} we get exponential tightness for the laws $\xi_n$ of $(\Sb_n,\Lb_n)$ if
\begin{enumerate}[(a)]
\item for each fixed $t\in\Q\cap [0,1]$ the sequence $\left((\Sb_n(t),\Lb_n(t))\right)_{n\in\N}$ is exponentially tight and
\item for every $\rho>0$,
	\begin{align*}
		\lim_{\delta\to 0}\sup_{n\in\N} \inv{n} \log \xi_n \left\{ f\colon w_f(\delta)\geq \rho\right\}& = -\infty
	\end{align*}
	where $w_f(\delta)=\sup_{\abs{t-s}\leq \delta} d\big(f(t),f(s)\big)$ is the modulus of continuity of $f$.
\end{enumerate}
Exponential tightness of $\left((\Sb_n(t),\Lb_n(t))\right)_{n\in\N}$ for every \emph{fixed} $t\in\Q\cap [0,1]$ is a direct consequence of Proposition~\ref{Thm:JointSanovCramer}. Moreover, for $0\leq s<t\leq 1$,
\begin{align*}
	d\left(\left(\Sb_n(t),\Lb_n(t)\right), \left( \Sb_n(s), \Lb_n(s) \right) \right) & \leq \frac{t-s}{n} \max_{j} X_j + \frac{t-s}{n}
\end{align*}
where the maximum on the right-hand side runs over the (finite) set of $j$ such that $\lfloor ns\rfloor\leq j\leq \lfloor nt \rfloor$. For any $\delta,\rho>0$ and $n\in\N$ it follows therefore that
\begin{align*}
	\inv{n}\log &\,\P\left\{\sup_{\abs{t-s}<\delta} d\left(\left(\Sb_n(t), \Lb_n(t)\right), \left( \Sb_n(s), \Lb_n(s) \right) \right)\geq \rho\right\} \\
	& \leq \inv{n}\log \P\left\{\frac{\delta}{n}\left(\max_{1\leq j\leq n} X_j+1 \right)\geq \rho \right\} = -\left(\frac{n\rho}{\delta} - 1\right) \log 2 \leq -\left(\frac{\rho}{\delta}-1 \right)\log 2.
	\end{align*}
	The right-hand side diverges to $-\infty$ as $\delta\to 0$. So condition (b)  also holds and $(\xi_n)_{n\in\N}$ is exponentially tight.
	\end{proof}
\end{lem}

\begin{lem}
	\label{Thm:LDPProjMDiml}
	For any fixed $0=t_0<t_1,\ldots,<t_m\leq 1$ the  sequence $(Z_n)_{n\in\N}$ of random variables defined by
	\begin{align*}
		Z_n & = \left(\left(\Sb_n \left(t_j\right) - \Sb_n \left(t_{j-1}\right),\Lb_n \left(t_j\right) - \Lb_n \left(t_{j-1}\right) \right)\right)_{j=1}^m \in \y^m
\end{align*}
satisfies a large deviations principle in $\y^m$ with good rate function given by
	\begin{align*}
	I_{t_1,\ldots,t_m} \left((x_1,\mu_1),\ldots,(x_m,\mu_m)\right) & = \sum_{j=1}^m \left(t_j - t_{j-1}\right) I_1\left(\frac{x_j}{t_j-t_{j-1}}, \frac{\mu_j}{t_j-t_{j-1}} \right). 
	\end{align*}

\begin{proof}
	Let $n$ be large enough so that $nt_j<nt_{j+1}-1$ and denote $\e=\y^m$. A direct calculation yields, for any $f=\left(\lambda_j,g_j\right)_{j=1}^m\in \e^*$,
\begin{align*}
	\lim_{n\to\infty} \inv{n} \log\E\, e^{nf(Z_n)} & = \sum_{j=1}^m \left(t_j - t_{j-1}\right)\Lambda_2\left(\frac{\lambda_j }{t_{j}-t_{j-1}},\frac{g_j}{t_{j}-t_{j-1}} \right) =: \Lambda_3(f)
\end{align*}
where $\Lambda_2(\lambda,g)=\log\,\E\left[\exp\left(\lambda X_1+g(\delta_{X_1})\right)\right]$.

By Corollary 4.6.14 of \cite{DemboZeitouni} this implies that the laws of $Z_n$ satisfy a large deviations principle on $\e$ with good rate function $\Lambda_1^*$ given by

\begin{align*}
	\Lambda_1^*\left(\left(x_j,\mu_j\right)_{j=1}^m \right) & = \sup \left\{ f\left(\left(x_j,\mu_j\right)_{j=1}^m - \Lambda_3(f) \colon f\in \e^* \right) \right\}\\
	& = \sum_{j=1}^d \left(t_j-t_{j-1}\right) \Lambda^*_2 \left( \frac{x_j}{t_j - t_{j-1}}, \frac{\mu_j}{t_j - t_{j-1}} \right).
\end{align*}
Since $I_1$ is convex it follows from the results of Section~\ref{Sec:JointSanovCramer} and Theorem 4.5.10(b) in \cite{DemboZeitouni} that $\Lambda_2^*=I_1$ and the lemma is proved.
\end{proof}
\end{lem}

\par\noindent The proof of Theorem~\ref{Thm:SamplePathJSC} now follows closely that of Theorem 1 of \cite{DemboZajic95}. An application of the contraction principle to the map $(z_1,\ldots,z_m)\Map (z_1,z_1+z_2,\ldots, z_1+\ldots z_m)$ yields the large deviations principle for the laws of $(\Sb_n(t_1),\Lb_n(t_1),\ldots,\Sb_n(t_m),\Lb_n(t_m))$ with good convex rate function given by
\begin{align*}
\wh{I}_{t_1,\ldots,t_m} \left((x_1,\mu_1),\ldots,(x_m,\mu_m)\right) & = \sum_{j=1}^m \left(t_j - t_{j-1}\right) I_1\left( \frac{x_j-x_{j-1}}{t_j-t_{j-1}}, \frac{\mu_j - \mu_{j-1}}{t_j-t_{j-1}} \right). 
\end{align*}
Applying the Dawson--G\"artner theorem as in the proof of Lemma~3 in \cite{DemboZajic95} yields a LDP for the laws of the pair process $(\Sb_n,\Lb_n)$ on $\c([0,1];\y)$ with good rate function

\begin{align*}
	I_2(\bm{x},\bm\mu) & = \sup_{t_1<\ldots<t_m} \wh{I}_{t_1,\ldots,t_m}\big(\bm{x}\left(t_1\right),\bm\mu \left(t_1 \right), \ldots, \bm{x}\left(t_m\right),\bm\mu\left(t_m\right)\big).
\end{align*}
Obviously $m_1\left(\bm\mu(t)\right)\neq \bm x(t)$ for some $t$ implies $I_2(\bm x, \bm\mu)=\infty$. Lemma 4 in \cite{DemboZajic95} then implies that $I_2$ is of the form (\ref{Eq:DefITwo}). This completes the proof of Theorem~\ref{Thm:SamplePathJSC}. \hfill \qedsymbol

\ \\
\par\noindent Finally let $(X_n)_{n\in\N}$, $(Y_n)_{n\in\N}$ be \emph{two} sequences of i.i.d. random variables of common law $\g_2$ and define $\Lb_n^X$, $\Lb_n^Y$, $\Sb_n^X$, $\Sb_n^Y$ analogously to (\ref{Eq:DefSPath}, \ref{Eq:DefLPath}). By Corollary 2.9 of \ts{Lynch--Sethuraman} \cite{LynchSethuraman87} we obtain the following

\begin{cor}
	\label{Thm:SamplePathJSCPair}
	The sequence of the laws of $(\Sb_n^X,\Lb_n^X,\Sb_n^Y,\Lb_n^Y)$ satisfies a large deviations principle on $\c\left([0,1],\y^2\right)$ with good rate function $I$ where for $(\bm{x},\bm{p},\bm{y},\bm{q})\in\y^2$,
	\begin{align*}
		I(\bm{x},\bm{p},\bm{y},\bm{q})  & = \begin{cases}
											\int_0^1 \left[H \left( \dot{\bm{p}} (s) \vert \g_2 \right) + H \left( \dot{\bm{q}} (s) \vert \g_2 \right)\right]\,\d s \quad & \text{if } (\bm{x},\bm{p}),(\bm{y},\bm{q})\in\e\\
						 +\infty & \text{otherwise.}
										\end{cases}
	\end{align*}
\end{cor}

\ \\


\section{Large Deviations for $\NC(n)$}
\label{Sec:LDPNCP}

Recall that, in the notation of Section~\ref{Sec:UnifNCP}, $\lambda_n=\inv{\wh{\tau}_n}\sum_{j=1}^{\tau_n} \delta_{Y_j}$ is the empirical measure of the blocks of a non-crossing partition picked uniformly at random. Define further $\sigma_n=m_1(\lambda_n) = \inv{\wh\tau_n} \sum_{j=1}^{\wh\tau_n} Y_j$. 

Let $\nu_n$ denote the law of $(\sigma_n, \lambda_n, \tau_n)$ on $\y\times [0,1]$. The main result of this section is the following.

\begin{thm}
	\label{Thm:LDP}
	The sequence $(\sigma_n,\lambda_n,\tau_n)_{n\in\N}$ satisfies a large deviations principle in $\y\times [0,1]$ with good convex rate function $J$ given by
\begin{align}
	J(m,\mu,t) & = \begin{cases}
	\log 4  -\inv{m} H(\mu) - \inv{m}\log\left(m-1\right) + \log\left( 1-\inv{m} \right)\quad & \text{if } m_1(p)=m=\inv{2t}\\
	+\infty & \text{otherwise.}
\end{cases}
\end{align}
\end{thm}

\par\noindent It is straightforward to verify that $J(m,\mu,t)=0$ if and only if $(m,\mu,t)= (2,\g_2,\inv{4})$. The following law of large numbers now follows immediately.

\begin{cor}
	\label{Thm:LLN}
	The empirical measure $\lambda_n$ of the block structure of a uniformly randomly chosen non-crossing partition converges weakly almost surely to the geometric distribution of parameter $\inv{2}$.
\end{cor}

\par\noindent We will first prove the upper bound, Proposition~\ref{Thm:LDPUpper} and then the lower bound, Proposition~\ref{Thm:LDPLower}. For both the following lemma is useful.

\begin{lem}
	\label{Thm:LogCmTrivial}
	The logarithmic asymptotics of the probability of $E_n$ are given by
	\begin{align}
		\label{Eq:LogCmTrivial}
		\lim_{n\to\infty} \inv{n}\log\P(E_n) & = 0.
	\end{align}
	\begin{proof}
		Writing $E_n$ in terms of the random walk $\Sigma$ that has ascents $X_1,X_2,\ldots$ and descents $Y_1,Y_2,\ldots$,
	\begin{align*}
		E_n & = \big\{\Sigma_{2n}=0,\, \Sigma_k>0\ \forall\, k<2n,\, \Sigma_{2n+1}=+1\big\}. 
		\intertext{Therefore, using the Markov property of $\Sigma$,}
		\P\big(E_n\big) & = \P\left\{\left.\Sigma_{2n+1}=+1\,\right\vert\, \Sigma_{2n}=0\right\}\cdot \P\left\{\Sigma_{2n}=0,\,\Sigma_k\geq 0\ \forall\, k<2n\right\}\\
	& = \inv{2}\cdot \frac{C_n}{4^n}
\end{align*}
because the second probability on the right is just that of running a simple random walk for $2n$ steps and obtaining a Dyck path. A direct computation using Stirling's formula \cite[p.64]{Durrett} yields that $\inv{n}\log C_n \longrightarrow 4$ as $n\to\infty$. Equation (\ref{Eq:LogCmTrivial}) follows.
	\end{proof}
\end{lem}

\par\noindent For any path $x\colon [0,1]\map \R$ with $x(0)=0$ and $x(t)-x(s)\geq t-s$ for all $t>s\geq 0$ we let $\tau(x)$ be the right-inverse of $x$ at 1, i.e.
\begin{align*}
	\tau(x) & = \inf\left\{s\in [0,1]\colon x(s)\geq 1\right\}.
\end{align*}	
If $\bm{p}(t)-\bm{p}(s)$ is a measure on \N\ of mass $t-s$ it follows that $m_1\left(\bm{p}(t)\right)-m_1\left(\bm{p}(s)\right)\geq t-s$. So the map $\e^2\map [0,1]$ given by $(\bm{x},\bm{p},\bm{y},\bm{q})\Map \tau(\bm{x}+\bm{y})$ is continuous.

\subsection{The Upper Bound}

We are now in a position to prove the large deviations upper bound. We will first give a bound via the process version and then show that this can be written in terms of the stated rate function

\begin{lem}
\label{Thm:LDUpperTemp}
	For every closed $F\subseteq \y\times [0,1]$ we have
	\begin{align}
		\label{Eq:LDPUpperTemp}
		\liminf_{n\to\infty} \inv{n} \log \nu_n(F) & \geq -2 \inf\left\{ I(\bm{x},\bm{p},\bm{y},\bm{q})\colon (\bm{x},\bm{p},\bm{y},\bm{q})\in \wh{F} \right\}
	\end{align}
where the (closed) subset $\wh{F}$ of $\y^2$ is defined by
\begin{align*}
	\wh{F} & = \left\{(\bm{x},\bm{p},\bm{y},\bm{q})\in\e^2 \colon \left( \inv{\tau}\bm{y}(\tau), \inv{\tau}\bm{q}(\tau),\tau \right)\in F, \bm{x}(\tau)=\bm{y}(\tau),\,\ \bm{x}(s)\geq \bm{y}(s)\,\forall s\leq \tau\right\}
	\end{align*}
and $\tau=\tau(\bm{x}+\bm{y})$.
\begin{proof}
	Recall that $\lambda_n = \inv{\tau_n}\Lb_{2n}^Y(\tau_n)$. Therefore,
	\begin{align*}
		\inv{n}\log\nu_n(F) & = \inv{n}\log\P\left\{\left(\inv{\tau_n}\Sb_{2n}^Y(\tau_n), \inv{\tau_n}\Lb_{2n}^Y(\tau_n), \tau_n\right)\in F;\ E_n\right\} - \inv{n}\log\P(E_n).
\end{align*}
By Lemma~\ref{Thm:LogCmTrivial} the second term on the right-hand side converges to 0. Further, $\tau_n=\inf\left\{\frac{k}{2n}\colon \inv{2n}(X_j+Y_j)\geq 1\right\}$, so that $\tau_n$ is the least integer multiple of $\inv{2n}$ less than $\tau(\Lb_{2n}^X+\Lb_{2n}^Y)$, with equality if and only if $\Sb_{2n}^X(\tau_n)+\Sb_{2n}^Y(\tau_n)=1$. This certainly holds on $E_n$, so we can write the event $E_n$ in terms of the $\Lb$, $\Sb$: for ease of notation we denote $\tau^{\Sb}:=\tau(\Sb_{2n}^X+\Sb_{2n}^Y)$. Then
	\begin{align*}
	E_n & = \left\{\Sb_{2n}^X(\tau_n) = \Sb_{2n}^Y(\tau_n)=\inv{2},\, \Sb_{2n}^X(s)\geq \Sb_{2n}^Y(s)\ \forall\, s\leq \tau_n\right\}\\
	& = \left\{\Sb_{2n}^X(\tau^{\Sb})) = \Sb_{2n}^Y(\tau^{\Sb})=\inv{2},\, \Sb_{2n}^X(s)\geq \Sb_{2n}^Y(s)\ \forall\, s\leq \tau^{\Sb},\ \tau^{\Sb}=\tau_n \right\}\\
	& \subseteq \left\{\Sb_{2n}^X(\tau^{\Sb})) = \Sb_{2n}^Y(\tau^{\Sb})=\inv{2},\, \Sb_{2n}^X(s)\geq \Sb_{2n}^Y(s)\ \forall\, s\leq \tau^{\Sb} \right\} =: \wt{E}_n.
	\end{align*}
Hence,
\begin{align*}
	\limsup_{n\to\infty}\inv{n}\log\nu_n(F) & \leq 2\limsup_{n\to\infty} \inv{2n}\log \P\left\{ \left(\inv{\tau^{\Sb}}\Sb_{2n}^Y(\tau^{\Sb}),\inv{ \tau^{\Sb}}\Lb_{2n}^Y(\tau^{\Sb}), \tau^{\Sb}\right)\in F;\ \wt{E}_n \right\}.
\end{align*}
Since $\tau^{\Sb}$ is a continuous function of $(\Sb_{2n}^X, \Lb_{2n}^X, \Sb_{2n}^Y, \Lb_{2n}^Y)$, the set on the right-hand side is closed in $\y^2$ and we can apply Corollary~\ref{Thm:SamplePathJSCPair} to obtain (\ref{Eq:LDPUpperTemp}).
\end{proof}
\end{lem}
\ \\

\par\noindent We now investigate the right-hand side of (\ref{Eq:LDPUpperTemp}). For any $(\bm{x},\bm{p},\bm{y},\bm{q})$ define new paths $\wt{p}$ and $\wt{q}$ by
\begin{align}
	\label{Eq:SLPath}
	\wt{p}(s) &= \begin{cases}
			\frac{s}{\tau(\bm{x}+\bm{y})} \bm{p}\left( \tau \left( \bm{x}+\bm{y} \right) \right) & \text{if } s\in\left[0,\tau \left( \bm{x}+\bm{y} \right)\right]\\
	\bm{p}\left( \tau \left( \bm{x}+\bm{y} \right) \right) + \left(s-\tau \left( \bm{x}+\bm{y} \right)\right) \g_2 \quad & \text{if } s\in\left[\tau \left( \bm{x}+\bm{y} \right),1\right]
		\end{cases}
\end{align}
and analogously $\wt{q}$, replacing $\bm{p}$ by $\bm{q}$. If further $\wt{x}(t)=m_1(\wt{p}(t))$ and $\wt{y}(t) = m_1(\wt{q}(t))$ for all $t$ then $\tau\left(\bm{x}+\bm{y}\right)=\tau\left(\wt{x}+\wt{y}\right)=:\tau$. Further $\left(\wt{x},\wt{p},\wt{y},\wt{q}\right) \in \wh{F}$ and
\begin{align*}
	I(\wh{x},\wh{p},\wh{y},\wh{q})  
	& = \tau\left(H \left( \inv{\tau}\bm{p} (\tau) \vert \g_2 \right) + H \left(\inv\tau \bm{q}(\tau) \vert \g_2 \right) \right).
\intertext{Moreover, by convexity of $H(\cdot\vert \g_2)$ (see \cite{DemboZajic95}, Lemma 4),}
	I(\bm{x},\bm{p},\bm{y},\bm{q})  
	& \geq \tau\left(H \left( \inv{\tau}\bm{p} (\tau) \vert \g_2 \right) + H \left(\inv\tau \bm{q}(\tau) \vert \g_2 \right) \right).
\end{align*}
It is clear that $\inv{\tau}\bm{p}(\tau)$, $\inv{\tau}\bm{q}(\tau)$ are probability measures, and that for every pair of probability measures $(p,q)$ such that $(m_1(p),p,\inv{2m_1(p)})\in F$ the corresponding straight-line path \big((\ref{Eq:SLPath}) with $\tau(\bm{x}+\bm{y}) = \inv{2m_1(p)}$\big) lies in $\wh{F}$. Therefore
\begin{align*}
	\inf_{\wh{F}} I & = \inf \left\{ \tau \left[H\left(p\vert\g_2\right) + H\left(q\vert\g_2\right) \right] \colon (m_1(q),q,\tau)\in F, m_1(p)=m_1(q)=\inv{2\tau}\right\}.
\end{align*}
On the other hand $H(q\vert\g_2) = -H(q)+ m_1(p)\log(2)$ and it is well-known that
\begin{align}
	\label{Eq:EntropyMax}
\sup \left\{\tau H(q)\colon m_1(q)=m\right\} & = \Theta(m):=\log(m-1) - m\log\left(1-\inv{m}\right).
\end{align}
Hence,
\begin{align}
	\label{Eq:RateSimplification}
	\inf_{\wh{F}} I & = \inf\left\{\log(2)-\tau H(p)- \tau\Theta \left( \inv{2\tau} \right) \colon (m_1(p),p,\tau)\in F, m_1(p)=\inv{2\tau}\right\}.
\end{align}
We have established the upper bound:

\begin{prop}
\label{Thm:LDPUpper}
	For every closed $F\subset\y\times [0,1]$,
	\begin{align}
		\label{Eq:LDPUpper}
		\liminf_{m\to\infty} \inv{m} \log \nu_m(F) & \leq -\inf\left\{ J(s,p,t) \colon (s,p,t)\in F \right\}.
	\end{align}

\end{prop}

\subsection{The Lower Bound}

We now turn to proving the lower bound. By the local nature of large deviations lower bounds (see \cite{DemboZeitouni}, identity (1.2.8) and the adjacent remarks) it is enough to prove the following.

\begin{prop}
\label{Thm:LDPLower}
	Fix $(m,\mu,t)\in\y\times [0,1]$ and $\rho_j>0$ for $j\in\{1,2,3\}$ and let $G= (m-\rho_2,m+\rho_2)\times B(\mu,\rho_1)\times (t-\rho_3,t+\rho_3)$. Then
	\begin{align}
		\label{Eq:LDPLower}
		\liminf_{n\to\infty} \inv{n} \log \nu_n\big(G\big) & \geq -J(m,\mu,t)
	\end{align}
where $B(\mu,r)$ denotes the ball in $\m_1$ of radius $r$, centred on $\mu$ with respect to $\beta$, the metric of  (\ref{Eq:DefBeta}) inducing weak topology.
\end{prop}

\begin{proof}
	We can assume that $m_1(\mu)=m=\inv{2t}$ since otherwise $J(m,\mu,t)=\infty$ and (\ref{Eq:LDPLower}) is trivial. From the definition of $\nu_n$ we have, as before,
	\begin{align*}
		\log\big(\nu_n(G)\big) &=  \log\P\left\{(\sigma_n,\lambda_n,\tau_n)\in G \right\} - \log\P(E_n).
\intertext{Recall that $\lim_{n\to\infty} \log\P(E_n)=0$. Moreover $E=\bigcup_r E_{n,r}$ where}
E_{n,r} &= \left\{ \sum_{j=1}^r X_j= \sum_{j=1}^r Y_j,\, \sum_{j=1}^k X_j\geq  \sum_{j=1}^k Y_j\ \forall\, k< r\,\ \tau_n=\frac{r}{2n}  \right\}.
\end{align*}
On $E_{n,r}$ we have $\frac{r}{2n}=\tau_n\in (t-\rho_3,m+\rho_3)$ and the condition $\sigma_n \in (m-\rho_2,m+\rho_2)$ is equivalent to $r\in \left(\frac{n}{m+\rho_2},\frac{n}{m-\rho_2}\right)$. Therefore
\begin{align*}
\P\left\{(\lambda_n,\sigma_n,\tau_n)\in G;E_n \right\} & = \sum_{r\in\i_n} \P\left\{\inv{r}\sum_{j=1}^r \delta_{Y_j}\in B(\mu,\rho_1);\,E_{n,r} \right\}
\intertext{where $\i_n=\N\cap\left(2n(t-\rho_3), 2n(t+\rho_3)\right)\cap \left( \frac{n}{m+\rho_2}, \frac{n}{m-\rho_2}\right)$. Fix now $w>0$ and let $N_1$ be large enough to have $nw>\frac{2}{\rho_1}$. Then if $r\in\uind{\i_n}{w}=\i_n\cap(wn,\infty)$ and $n\geq N_1$,}
	\beta\left(\inv{r}\sum_{j=1}^r \delta_{Y_j},\inv{r}\sum_{j=1}^{r-1} \delta_{Y_j}\right) & = \inv{r}<\frac{\rho_1}{2}.
	\intertext{Using independence of the $X_j,Y_j$ and the fact that $\P(Z=a)=\P(Z>a)$ for any $Z$ with law $\g_2$,}
	\P\left\{(\lambda_n,\sigma_n,\tau_n)\in G;\, E_n\right\} & \geq \inv{4} \sum_{r\in\uind{\i_n}{w}} \P\left\{\inv{r}\sum_{j=1}^{r-1} \delta_{Y_j}\in B\left(\mu,\frac{\rho_1}{2}\right),\, E_{n,r}',\,\tau_n=\frac{r}{2n}  \right\}\\
 & \geq \inv{4} \sum_{r\in\uind{\i_n}{w}} \P\left\{\inv{r}\sum_{j=1}^{r} \delta_{Y_j}\in B\left(\mu,\frac{\rho_1}{4}\right),\, E_{n,r}',\,\tau_n=\frac{r}{2n}  \right\}
\end{align*}
for $n\geq 2N_1$. Here, 
\begin{align*}
	E_{n,r}'	&=\left\{\inv{r}\sum_{j=1}^r X_j\geq n,\, \inv{r}\sum_{j=1}^r Y_j\geq n,\ \sum_{j=1}^a X_j\geq \inv{r}\sum_{j=1}^a Y_j
\ \forall a<r,\tau_n=\frac{r}{2n}\right\}. 
\intertext{Recall that $\lambda_n=\inv{\tau_n}\Lb_{2n}^Y(\tau_n)$, that $\tau_n-\inv{n}\leq \tau^{S}:=\tau(\Sb_{2n}^X+\Sb_{2n}^Y)\leq \tau_n$ and that the $\Sb$-processes are increasing in time. It follows that}
	E_{n,r}' & = \left\{\Sb_{2n}^X(s)\geq \Sb_{2n}^Y(s)\,\forall\, s<\tau_n,\ \Sb_{2n}^X(\tau_n)\geq \inv{2},\, \Sb_{2n}^Y(\tau_n)\geq \inv{2},\tau_n=\frac{r}{2n} \right\}\\
	& \supseteq \left\{\Sb_{2n}^X(s)\geq \Sb_{2n}^Y(s)\,\forall\, s<\tau^{\Sb},\ \Sb_{2n}^X(\tau^{\Sb})\geq \inv{2},\, \Sb_{2n}^Y(\tau^{\Sb})\geq \inv{2} ,\tau_n=\frac{r}{2n}\right\}.
\end{align*}
Denote by $\wt{E}_{n,r}$ the latter event and define $\wt{E}_n=\bigcup_{r\in\uind{\i}{w}_{n}} E_{n,r}$. We obtain
\begin{align*}
	\P\left\{\left(\sigma_n,\lambda_n,\tau_n\right)\in G;\ E_n\right\} &\geq \inv{4}\P\left\{\inv{\tau_n}\Lb_{2n}^Y\in B\left(\mu,\frac{\rho_1}{4}\right);\ \wt{E}_{n}\right\}.
\intertext{Let now $N_2$ be large enough such that $n\geq N_2$ implies $n>\frac{2}{w}\vee\frac{2}{\rho_3}\vee \frac{4w^2}{\rho_1}\vee \frac{8}{w\rho_1}$ and $\inv{m_1+2\rho_2}-\inv{n}<\inv{m_1+\rho_2}$. Then}
	\beta \left( \inv{\tau_n}\Lb_{2n}^Y(\tau_n)  - \inv{\tau^{\Sb}} \Lb_{2n}^Y (\tau^{\Sb}) \right) & \leq \beta \left( \inv{\tau_n}\Lb_{2n}^Y(\tau_n)  - \inv{\tau^{\Sb}} \Lb_{2n}^Y (\tau_n) \right) \\
	&\quad+ \beta \left( \inv{\tau^{\Sb}}\Lb_{2n}^Y(\tau_n)  - \inv{\tau^{\Sb}} \Lb_{2n}^Y (\tau^{\Sb}) \right)\\
	& \leq \abs{\inv{\tau_n}-\inv{\tau^{\Sb}}} \beta \left(\Lb_{2n}^Y(\tau_n),0\right) + \frac{\abs{\tau_n-\tau^{\Sb}}}{\tau^{\Sb}}\\
	& < \inv{nw}<\frac{\rho_1}{8}
\end{align*}
and further, using the fact that $\abs{\tau_n-\tau^{\Sb}}<\inv{n}$ repeatedly,
\begin{align*}
	\left\{\tau^{\Sb}\in\i^\tau \right\} & \subseteq \{\tau_n>\frac{w}{2},\,\tau\in (t-\rho_3,t+\rho_3)\cap \left(\inv{m+\rho_2}, \inv{m-\rho_2}\right)
\end{align*}
where $\i^\tau=(w,\infty)\cap \left(t-\frac{\rho_3}{2},t+\frac{\rho_3}{2}\right) \cap \left(\inv{2m+4\rho_1}, \inv{2m-2\rho_1}\right)$. It follows that
\begin{align}
\label{Eq:PfLDPLower}
\P\left\{\inv{\tau_n}\Lb_{2n}^Y\in B\left(\mu,\frac{\rho_1}{4}\right); \wt{E}_{n}\right\} & \geq \P\left\{ \inv{\tau^{\Sb}} \Lb_{2n}^Y \left(\tau^{\Sb} \right)\in B\left(\mu,\frac{\rho_1}{8}\right),\tau^{\Sb}\in \i^\tau, \left(\Sb^X_{2n},\Sb^Y_{2n}\right) \in \i^{\Sb}_w \right\}.
\end{align}
Here,
\begin{align*}
\i^{\Sb}_w&=\{(x,y)\colon x(s)> y(s)-w\,\forall s<\tau(x+y),\, x(\tau(x+y)),\, y(\tau(x+y))>\inv{2}-w\}.
	\end{align*}
The right-hand side of (\ref{Eq:PfLDPLower}) is of the form $(\Sb_{2n}^X,\Sb_{2n}^Y,\Sb_{2n}^X, \Sb_{2n}^X)\in U$ for an open subset $U$ of $\c\left([0,1];\y^2\right)$. So we can apply Corollary~\ref{Thm:SamplePathJSCPair}, then let $w\to 0$ and obtain
\begin{align}	
	\label{Eq:LBProcTerms}
	\liminf_{n\to\infty}\inv{n}\log\nu_n(G) & \geq -2\inf\left\{I(x,p,y,q)\colon  \beta\left(\inv{\tau}q(\tau),\mu\right)<\frac{\rho_1}{8},\tau\in\i^{\tau}, (x,y)\in \i^{\Sb}\right\}
\end{align}
where $\tau:=\tau(x+y)$ and $\i^{\Sb}=\{(x,y)\colon x(s)\geq y(s)\,\forall s<\tau,\, x(\tau)= y(\tau)=\inv{2}\}$. Let $(\bm{p},\bm{x})\in\e$ be such that $\bm{x}(s) = m_1(\bm{p}(s))\ \forall\, s$, for any $s\in [0,t)$ the inequality $x(s)\leq s\mu$ holds and $x(t)=\inv{2}$. Define further $\wt p\colon [0,1]\map \m_+(\N)$ by
\begin{align*}
	\wt{p}(s) & = \begin{cases}
			s\mu &\text{if } s\in [0,t]\\
			t\mu + (s-t)g_2\quad & \text{if } s\in [t,1]
	\end{cases}
\intertext{($g_2\sim\g_2$) and $\wt{y}(t)=m_1(\wt{p})$. Then $(x,p),\,(\wt{y},\wt{q})\in\e$ and $\tau(x+\wt{y})=t\in\i^\tau$. By construction $(x,\wt{y})\in\i^{\Sb}$. Moreover, $\int_0^1 H\left(\wt{q}(s)\vert\g_2\right)\,\d s=t H(\mu\vert \g_2)$ and  $\int_0^1 H(p(s)\vert\g_2)\,\d s\geq t H(\inv{t}\bm{p}(t)\vert\mu)$. So by (\ref{Eq:LBProcTerms}), and the same argument as for (\ref{Eq:RateSimplification}),}
\liminf_{n\to\infty}\inv{n}\log\nu_n(G) & \geq -2\inf \left\{ I\left(\bm{x}, \bm{p}, \wt{y},\wt{q}\right) \colon x(s)\leq s\mu,\, x(t)=\inv{2}\right\}\\
 & = -2\inf\left\{ t\left(H\left(\mu\vert\g_2\right)+H\left(p\vert\g_2\right) \right)\colon m_1(q)=\inv{2} \right\}\\
	& \geq - J(m,\mu,t).
\end{align*}
This concludes the proof of the lower bound, and hence of Theorem~\ref{Thm:LDP} for $(\nu_n)_{n\in\N}$.
\end{proof}

\section{A Formula for the Maximum of the Support}
\label{Sec:SpectralEdgeFormula}

In this section we apply our large deviations result to a problem from free probability theory. Fix a compactly supported probability measure $\mu$. Its \emph{Cauchy transform} is the analytic function $G_\mu$ where for $z\in\C\setminus\spt\mu$,
\begin{align*}
	G_\mu(z) & = \int \frac{\mu(\d t)}{z-t}.
	\intertext{The function $G_\mu$ is analytic on $\C\setminus\spt(\mu)$ and locally invertible on a neigbourhood of $\infty$. Its inverse, $K_\mu$ is meromorphic around zero, where it has a simple pole of residue 1. Removing this pole we obtain an analytic function}
	R_\mu(z) &= K_\mu(z) - \inv{z} = \sum_{n=0}^\infty k_{n+1}(\mu) z^n.
\end{align*}
The function $R_\mu$ is called the \emph{R-transform} of $\mu$ while its coefficients $\left(k_n\left(\mu\right)\right)_{n\in\N}$ are called the \emph{free cumulants} of $\mu$. Given that $\mu$ has compact support it is determined by its R-transform. So, given an R-transform we can, at least in theory, obtain the corresponding probability measure. However in order to do so one needs to find the functional inverse of $R(z)+\inv{z}$ for which a closed-form expression may not exist. Using the large deviations principle of Section~\ref{Sec:LDPNCP} we can deduce the right edge of the support of $\mu$,  provided that the free cumulants are non-negative.

The problem of determining a measure from its R-transform occurs in free probability: if $a_1,a_2$ are free non-commutative random variables of law $\mu_1,\mu_2$ respectively then the law $\mu$ of $a_1+a_2$ has the property that $k_n(\mu)=k_n(\mu_1)+k_n(\mu_2)$ and the law $\nu$ of $\lambda a_1$ has $k_n(\nu)=\lambda k_n(\mu_1)$ for any $\lambda\in\R$. This linearity property allows the computation of laws of free random variables, similarly to the moment generating function in commutative probability theory. For background on free probability theory see for example \cite{VDN, HiaiPetz} and the survey of probabilistic aspects of free probability theory \cite{BianeFPP}. 

Because the R-transform determines the underlying probability measure one might still hope to recover some information about the measure, for example about the support, even when the Cauchy transform cannot be obtained explicitly. The special case where the underlying law is a free convolution of a semicircular law with another distribution has been studied extensively by P.~\ts{Biane} \cite{Biane97}.

In this section we describe how the maximum of the support of $\mu$ can be deduced from the free cumulants.

Combinatorial considerations of the way R- and Cauchy transform are related \cite{NicaSpeicher, Speicher94} give rise to the \emph{free moment-cumulant formula}:
\begin{align}
	\label{Eq:MomentCumulant}
	\int t^n \mu (dt) & = \sum_{\pi\in\NC(n)} \prod_{j=1}^\infty k_j(\mu)^{B_j(\pi)}
\intertext{where $B_j(\pi)$ is the number of blocks of size $j$ in $\pi$.Our starting point is the observation that the edge of the support of a measure can be deduced from the logarithmic asymptotics of its moments: namely if $\rho_\mu$ is the maximum of the support of $\mu$ then}
	\label{Eq:EdgeSptMoments}
	\log \rho_\mu & = \limsup_{n\to\infty} \inv{n}\log\int t^n\mu(\d t).
\end{align}
Suppose for the moment that all cumulants are positive (which is indeed the first case we will consider, in Section~\ref{Sec:CumPve}). Then we can re-write (\ref{Eq:MomentCumulant}) as the expectation of an exponential functional of a uniformly random non-crossing partition. Namely, if $\theta\colon\N\map\R$ is given by $\theta_j:= \log k_j$, and $\wh{\E}_n=\E(\cdot\vert E_n)$ (recall that $E_n$ denotes the event we conditioned on in Section \ref{Sec:UnifNCP}) and $C_n$, the $n^\text{th}$ Catalan number, is the cardinality of $\NC(n)$) then
\begin{align*}
	\int t^n \mu(\d t) & = \sum_{\pi\in\NC(n)} \exp\left(\sum_{j=1}^\infty \log(k_j) B_j(\pi)\right) = C_n\,\wh\E_n\left[e^{2n\tau_n\,\langle\theta,\lambda_n\rangle}\right].
\end{align*}
In Section~\ref{Sec:CumPve} we evaluate the logarithmic asymptotics of this expectation by Varadhan's Lemma, using the large deviations principle we have proved above.
\par Using the fact that $\lim_{\epsilon\to 0} \epsilon\log\epsilon=0$ one might suppose that a similar result will still hold when some of the cumulants are allowed to be zero. This is indeed the case and we will prove this in Section~\ref{Sec:CumNNve}.

\begin{rmk}
	\label{Rmk:ShiftFirstCumulant}
	For $\gamma\in\R$ the shift operation given by $S_\gamma(\mu)(E)=\mu(\{x-\gamma\colon x\in E\}$ shifts the maximum of the support by $\gamma$ to the right. Also $S_\gamma(\mu)=\mu\boxplus\delta_\gamma$ which leaves all cumulants invariant, except for the first which is incremented by $\gamma$. So we can always take the first cumulant to be anything we like. 
\end{rmk}

\subsection{All Free Cumulants Positive}
\label{Sec:CumPve}

\par\noindent We first consider the case where all free cumulants are positive. Examples include the free Poisson distribution.

\begin{thm}
	\label{Thm:EdgePositive}
	Let $\mu$ be a compactly supported probability measure on $[0,\infty)$ such that its free cumulants $(k_j)_{j\in\N}$ all positive. Then the right edge $\rho_\mu$ of the support of $\mu$ is given by
	\begin{align}
		\label{Eq:EdgePositive}
		\log\rho_\mu & = \sup\left\{\inv{m_1(p)}\sum_{m=1}^\infty p_m\log\left(\frac{k_m}{p_m}\right) + \frac{\Theta\big(m_1(p) \big)}{m_1(p)} \colon p\in\m_1^1(\N)\right\}
	\end{align}	
	where $\m_1^1(\N)=\{p\in\m_1(\N)\colon m_1(p)<\infty\}$ is the set of probability measures on \N\ with finite mean and $\Theta$ was defined in (\ref{Eq:EntropyMax}).
\end{thm}

\par\noindent This variational problem can often be solved by Lagrange multipliers or similar methods. Some examples are given below.

\begin{rmk}
	\label{Rmk:VaradhanSRF}
	Equation (\ref{Eq:EdgePositive}) looks somewhat similar to \ts{Varadhan}\emph{'s spectral radius formula} \cite[Exercise 3.1.19]{DemboZeitouni}, giving the spectral radius of a (deterministic) $N\times N$ matrix in terms of its entries. Namely, let $B=\left(b_{ij}\right)_{i,j=1}^N$ be irreducible and have strictly positive entries then the spectral radius (absolutely largest eigenvalue) $\rho_B$ of $B$ is given by
\begin{align}
	\label{Eq:VaradhanSRF}
	\log \rho_B &= \sup\left\{\sum_{i,j=1}^N q(i,j) \log\left(\frac{b(i,j)}{q_f(j\vert i)} \right) \colon q\in \m_1 \left( \ul{N}\times \ul{N}\right), \sum_{j=1}^N q(\cdot,j) = \sum_{j=1}^n q(j,\cdot) \right\}
\end{align}
where $q_f(j\vert i) = \frac{q(i,j)}{\sum_r q(i,r)}$. Despite the apparent formal similarities we do not seem to be able to relate this formula to ours. This is because the free cumulants of a deterministic matrix are given by a very complicated function of its entries.
\end{rmk}

\begin{proof}[Proof of Theorem~\ref{Thm:EdgePositive}] 
By Stirling's formula, $\inv{n}\log C_n \longrightarrow \log 4$ as $n\longrightarrow\infty$, so that
\begin{align}
	\label{Eq:PlusCatalan}	
	\limsup_{n\to\infty} \int t^{n}\,\mu(\d t) & = \log 4 + \limsup_{n\to\infty} \inv{n}\log \wh\E_n\left[e^{n g(\lambda_n,\tau_n)} \right]
\end{align}
where $g\colon \m_1(\N)\times [0,1]\map\R$ is defined by $g(\mu,t)=2t\langle\theta,\mu\rangle$.

It is a direct application of the contraction principle, Theorem 4.2.1 in \cite{DemboZeitouni}, that $(\lambda_n,\tau_n)_{n\in\N}$ satisfies a large deviations principle on $\m_1(\N)\times [0,1]$ with rate function $\wt{J}_{13}$ given by $\wt{J}_{13}(\mu,t)=J(\mu,m_1(\mu),t)$.

Suppose first that the sequence $(k_n)_{n\in\N}$ is bounded by $K\in (0,\infty)$. Then $g$ is continuous and bounded, with norm $\norm{g}_\infty\leq 2\log K$. So for \emph{any} $\gamma>1$,
\begin{align*}
	\limsup_{n\to\infty} \inv{n}\log \wh{\E}_n\left[e^{n\gamma g(\tau_n, \lambda_n)} \right] & \leq 2\gamma\log(K)<\infty.
\end{align*}
Hence the moment condition for Varadhan's Lemma (\ts{Dembo--Zeitouni} \cite{DemboZeitouni}, Theorem 4.3.1) applies and so
\begin{align*}
	\lim_{n\to\infty} \inv{n} \log \wh{\E}_n \left[ e^{n g(\lambda_n,\tau_n)}\right] & = \sup\left\{g(\mu,t) - \wt{J}_{13}(\mu,t)\colon (\mu,t)\in \m_1(\N) \times [0,1] \right\}.
\end{align*}
Let $\wh{\rho}_\mu$ denote the left-hand side above and note that $\rho_\mu=\wh{\rho}_\mu +\log 4$. So
\begin{align*}
\log(\rho_\mu) & = \sup\left\{\inv{m} \sum_{n=1}^\infty p_n \log k_n +\inv{m} H(p) - 2\log\left(1-\inv{m}\right) \colon m_1(p)=m \right\}
\end{align*}
which is (\ref{Eq:EdgePositive}).

We now turn to the general case, that is, we remove the assumption that the sequence of free cumulants is bounded. Because $\mu$ is compactly supported, its R-transform is analytic on a neighbourhood of zero, by Theorem 3.2.1 in \ts{Hiai--Petz} \cite{HiaiPetz}. So there exist $\Gamma,R\in (0,\infty)$ with $k_n\leq \Gamma R^n$ for all $n\in\N$. Define the dilation operator of scale $\inv{R}$ by $D_{R^{-1}} (\mu)(A) =\mu\big(R^{-1}A\big)$, \big(where $tA=\{tx\colon x\in A\}$\big) and let $\wh{k}_n=R^{-n}k_n$ be the $n^\text{th}$ cumulant of $D_{R^{-1}}(\mu)$. The sequence $\left(\wh{k}_n\right)_{n\in\N}$ is bounded, so the above applies to $\rho_{D_{R^{-1}}(\mu)}=\frac{\rho_\mu}{R}$. In particular,
\begin{align*}
\log\rho_{D_{R^-1}(\mu)} & = \sup\left\{\inv{m} \sum_{n=1}^\infty p_n \log \wh{k}_n +\inv{m} H(p) +\inv{m}\, \Theta(m)\colon m_1(p)=m \right\}\\
& =\sup\left\{\inv{m} \sum_{n=1}^\infty p_n \log k_n   +\inv{m} H(p) +\inv{m}\, \Theta(m)\colon m_1(p)=m \right\} - \log(R)
\end{align*}
which completes the proof of Theorem~\ref{Thm:EdgePositive}.
\end{proof}

\subsection{Non-Negative Free Cumulants}
\label{Sec:CumNNve}

We now consider non-commutative random variables of which all free cumulants are non-negative but some of them are allowed to take the value zero. We will denote by $L$ the set of $n\in\N$ such $k_n\ne 0$. As a prominent example we mention the centred semicircle distributions, where $L=\{2\}$. 

It turns out that the variational formula (\ref{Eq:EdgePositive}) still holds, provided we follow the convention that $0\log 0=0$.

\begin{thm}
	\label{Thm:EdgeNonnegative}
	Let $\mu$ be a compactly supported probability measure whose free cumulants $(k_n)_{n\in\N}$ are all non-negative. Then the maximum of the support $\rho_\mu$ of $\mu$ is given by
	\begin{align}
		\label{Eq:EdgeNonnegative}
		\log\left(\rho_\mu \right) & = \sup\left\{\inv{m_1(p)}\sum_{n\in L} p_n\log\left(\frac{k_n}{p_n}\right) - \frac{\Theta(m_1(p))}{m_1(p)} \colon p\in\m_1^1(L)\right\}
	\end{align}	
	where we $\m_1^1(L)$ denotes the set of $p\in\m_1^1(\N)$ such that $p(L^c)=0$.
\end{thm}

\begin{proof}
	Since the set $\left\{p\in\m_1(\N)\colon m_1(L^c)=0\right\}$ is closed the direction `$\leq$' in (\ref{Eq:EdgeNonnegative}) follows directly from Exercise 2.1.24 in \ts{Deuschel--Stroock} \cite{DeuschelStroock}. So we only need to show that the logarithm of the maximum of the support of our measure is bounded below by the variational formula. Let $p$ be the free Poisson distribution with parameter 1 and recall that $p$ has support $[0,4]$. For $\epsilon>0$ let $\nu_\epsilon = D_{\epsilon^{-1}}(p)$, the $\epsilon$-dilation of $p$ (see the proof of Theorem~\ref{Thm:EdgePositive}). Then $k_n(\nu_\epsilon)=\epsilon^n$. By the remarks after Example 3.2.3 in \cite{HiaiPetz} (page 98) the maximum of the support of $\mu_\epsilon:=\mu \boxplus \nu_\epsilon$ is no bigger than the sum of those of $\mu$ and $\nu_\epsilon$. Moreover Theorem~\ref{Thm:EdgePositive} applies to $\mu_\epsilon$ so that
\begin{align*}
	\rho_\mu+4\epsilon & \geq \rho_{\mu_\epsilon}\\
	& = \sup\left\{\inv{m_1(p)}\sum_{n=1}^\infty p_n \log \left( \frac{k_n + \epsilon^n}{p_n} \right) + \frac{\Theta(m_1(p))}{m_1(p)} \colon p\in \m_1^1(\N) \right\}\\
	& \geq \sup\left\{\inv{m_1(p)}\sum_{n=1}^\infty p_n \log \left( \frac{k_n + \epsilon^n}{p_n} \right) + \frac{\Theta(m_1(p))}{m_1(p)} \colon p\in \m_1^1(L) \right\}\\
	& \geq \sup\left\{\inv{m_1(p)}\sum_{n=1}^\infty p_n \log \left( \frac{k_n}{p_n} \right) + \frac{\Theta(m_1(p))}{m_1(p)} \colon p\in \m_1^1(L) \right\}
\end{align*}
using the fact that $\epsilon>0$. Letting $\epsilon$ tend to zero yields the  `$\geq$' direction of (\ref{Eq:EdgeNonnegative}).
\end{proof}

\section{Examples}

We conclude with a few examples where our formula can be applied. The main requirement, that the free cumulants be non-negative, is satisfied in a wide range of cases. 

\begin{ex}
	As a warm-up let us consider two (known) examples where the variational problem can be solved to give an explicit formula for the maximum of the support. The simplest example is the \emph{centred semicircle law} of radius $r$ given by
\begin{align*}
	\sigma_r(\d t) & = \frac{2}{\pi r^2}\, \sqrt{r^2- t^2}\,\one_{[-r,r]} \d t.
\end{align*}
Then, in the notation of Section~\ref{Sec:SpectralEdgeFormula}, $L=\{2\}$ and $k_2(\sigma_r)=\frac{r^2}{4}$. The only probability measure on $L$ is $\delta_2$ which has $m_1(\delta_2)=2$. Therefore,
\begin{align*}
	\log \rho_{\sigma_r} &= \inv{2}\log k_2 + \inv{2}\Theta(2)\\
						& = \log\left(\frac{r}{2}\right) + \inv{2}\left(2\log\left(1-\inv{2}\right)\right) = \log(r).
\intertext{Next let $\lambda\geq 1$ and consider the \emph{free Poisson distribution} $p_\lambda$ with parameter $\lambda$, i.e.,}
		p_\lambda(\d t) & = \inv{2\pi t}\,\sqrt{4\lambda-\left(t-1-\lambda\right)^2}\, \one_{[(1-\sqrt{\lambda})^2,(1+\sqrt{\lambda})^2]}(t)\, \d t.
	\intertext{The free cumulants are given by $k_n=\lambda$ for all $n\in\N$ and therefore}
		\log \rho_{p_\lambda} &=  \sup\left\{ 2\tau\log(\lambda) + 2\tau H(p) + 2\tau \Theta\left(\inv{2\tau} \right) \colon m_1(p)=\inv{2\tau}\right\} \\
	& = 2\sup_{\tau\leq \inv{2}}\left[\tau \log\lambda + 2\tau\Theta \left(\inv{2\tau}\right) \right].
	\end{align*}
	Putting $\Psi_\lambda(\tau)=\tau \log\lambda + 2\tau\Theta \left(\inv{2\tau}\right)$ we easily verify that $\Psi_\lambda'(\tau^*)=0$ for $\tau^*=\frac{\sqrt{\lambda}}{2\left(\sqrt{\lambda}+1\right)}$ and that this critical point is the absolute maximum of $\Psi_\lambda$ on $\left[0, \inv{2}\right]$. Another direct computation yields $\log \rho_{p_\lambda} = 2\Phi_\lambda(\tau^*)=2\log\left(1+\sqrt\lambda \right)$, i.e. $\rho_{p_\lambda}=\left(1+\sqrt{\lambda}\right)^2$.
\end{ex}
%
%


\begin{ex}
	Let us consider $\mu= p \boxplus u$ where $p$ is the free Poisson law of parameter 1 and $u$ is the uniform distribution on $[-1,1]$. This corresponds, for example, to the limiting distribution of $T_N^* T_N + \diag(\rho_1,\ldots,\rho_N)$ where $T_N$ is an $N\times N$ real random matrix with i.i.d. entries of mean 0 and variance 1 and all moments bounded and $\rho_N(j)=j-1-\frac{N}{2}$. The R-transform of $\mu$ is given by
\begin{align*}
	R_\mu(z) = R_p(z)+R_u(z) = \frac{1}{1-z}+\coth(z)-\inv{z}
\end{align*}
which cannot be inverted explicitly. We obtain an implicit equation for the maximum of the support of $\mu$, i.e. the limiting largest eigenvalue:
\begin{align*}
	\rho_\mu & =\frac{\pi(m-1)}{m \gamma}
	\intertext{where $(\gamma,m)$ is the unique pair of positive reals satisfying the equations}
	\inv{m-1} & = \frac{\gamma}{1-\gamma}+\coth(\gamma)\\
	\frac{\lambda(m-1)}{1-\gamma} + (m-1)\coth(\gamma) & = \frac{m-1}{\gamma}  + \frac{\gamma^2+(1-\gamma)^2}{m\gamma(1-\gamma)^2}+\frac{\gamma}{m}\, \left(1-\coth^2(\gamma)\right).
\end{align*}
\end{ex}

\subsection{Freely Infinitely Divisible Distributions}

Let $\mu$ be \textit{freely infinitely divisible}. That is, for every $n\in\N$ there exists a compactly supported probability law $\mu_n$ such that $\mu$ is the $n$-fold free convolution of $\mu_n$ with itself:
\begin{align}
	\notag
	\mu & = \underbrace{\mu_n\boxplus\ldots\boxplus\mu_n}_{n\text{ times}}.
\intertext{Freely infinitely divisible probability measures have been studied by \ts{Barndorff-Nielsen} -- \ts{Thorbj\o rnsen} \cite{BNT_1, BNT_2}. Many of their properties are non-commutative analogues of those enjoyed by classical infinitely divisible distributions, for example they lead to the concept of free L\'evy processes. There exists an analogue of the L\'evy-Khintchine representation, a version of which is given in \cite{HiaiPetz}, where Theorem 3.3.6 states that $\mu$ is freely infinitely divisible if and only if there exist $\alpha\in\R$ and a positive finite measure $\nu$ with compact support in \R\ such that the R-transform $R_\mu$ of $\mu$ can be written, for $z$ in a neighbourhood of $(\C\setminus\R)\cup\{0\}$, as}
	\label{Eq:LevyKhintchine}
	R_\mu(z) & = \alpha + \int\frac{z}{1-xz}\, \nu(\d x).
\end{align}
We call $\nu$ the \textit{free L\'evy--Khintchine measure} associated to $\mu$.
By Remark~\ref{Rmk:ShiftFirstCumulant} we lose no generality by setting $k_1(\mu)=\alpha=0$. Setting $m_0(\nu):=\nu(\R)$ we can express the cumulants of $\mu$ in terms of the sequence $\left(m_n(\nu)\right)_{n\geq 0}$ by observing that $k_n(\mu)=m_{n-2}(\nu)$ for $n\geq 2$.
\par So if $\mu$ is freely infinitely divisible and the moments of its free L\'evy-Khintchine measure are all non-negative the variational formula for the maximum of the support of $\mu$ from Theorem~\ref{Thm:EdgeNonnegative} applies.

\subsection{Series of Free Random variables}

Let $\xi_1,\xi_2,\ldots$ be a sequence of free self-adjoint random variables of identical distribution $\mu_1$ and consider the series
\begin{align*}
	\xi & = \sum_{n=1}^\infty n^{-\beta} \xi_n
\end{align*}
where $\beta>0$ is chosen large enough for the series to converge in the operator norm. Let $k_n(\mu_1)$ denote the free cumulants of $\mu_1$ then the R-transform $R_\xi$ of $\xi$ is given by
\begin{align*}
	R_\xi(z) & = \sum_{n=1}^\infty n^{-\beta} R_{\xi_1} \left(n^{-\beta}z\right) = \sum_{n=1}^\infty n^{-\beta} \sum_{r=0}^\infty k_r(\mu_1) \left(n^{-\beta}z\right)^r.
\intertext{Let $U$ be a neighbourhood of zero where $R_{\xi_1}$ is analytic then we have absolute convergence on $U$ and hence may interchange the order of the two summations:}
R_\xi(z) & = \sum_{r=0}^\infty \sum_{n=1}^\infty n^{-r\beta}  k_r(\mu_1) z^r = \sum_{r=0}^\infty \zeta(\beta r)  k_r(\mu_1) z^{r-1} 
\end{align*}
where $\zeta$ denotes the Riemann zeta function. So we conclude that the free cumulants of $\xi$ are given in terms of those of $\xi_1$ by $k_n=\zeta(\beta n)\,k_n(\mu_1)$.

It may not be possible to locally invert the corresponding analytic function $R_0$ in closed form. In this case our formula comes in useful and we obtain:

\begin{cor}
	Suppose that the free cumulants of $\mu_1$ are all non-negative. Then the right edge $\rho_{0}$ of the support of the law of the series $\xi_0$ is given by
\begin{align}
	\label{Eq:SeriesVP}
	\log\left(\rho_0 \right) & = \sup\left\{\inv{m_1(p)}\sum_{n=1}^\infty p_n\log\left(\frac{\zeta(\beta n) \uind{k}{0}_n}{p_n}\right) - \frac{\Theta(m_1(p))}{m_1(p)} \colon p\in\m_1^1(L)\right\}.
\end{align}
\end{cor}

\par\noindent In some cases we can solve this variational problem and obtain a more or less explicit formula for the maximum of the support.

\begin{ex}
	Suppose $\mu_1$ is the free Poisson distribution of parameter $\lambda\geq 1$. We set $\beta=2$ and study $\sum_n n^{-2} \xi_n$ where the $\xi_n$ are free and all distributed according to the free Poisson law. The corresponding R-transform is 
\begin{align*}
	R(z) & =\frac{\lambda\left(1-\sqrt{z}\cot\left(\sqrt{z}\right)\right)}{2z}
\end{align*}
for which no closed-form inverse exists. However there is a unique maximiser for the corresponding variational problem (\ref{Eq:SeriesVP}), given by $p_n=\frac{\lambda\zeta(2n)}{Z} e^{m t n}$ and determined by its mean $m$. That mean is given implicitly by
\begin{align*}
	\lambda(m-1)-2 & = \sqrt{4\lambda m^2-2\lambda m - 2(\lambda-2)}\, \cot\left( \frac{\sqrt{4\lambda m^2-2\lambda m - 2(\lambda-2)}}{\lambda(m-1)}\right)
\intertext{which has a unique solution $m_*$ in the relevant interval. The right edge is therefore given by}
	\rho &= \log \frac{\lambda^2 m_*^2(m_*-1)}{4\lambda m_*^2-2\lambda m_* -2(\lambda -2)}.
\end{align*}
The choice $\lambda=1$ corresponds to the square integral of a free Brownian bridge which has been studied in \cite{FuncBB}. 
\end{ex}

\par\noindent Another example, where the $\xi_n$ are distributed according to the \emph{commutator} of the standard semicircle law with itself, can also be found in \cite{FuncBB}. The commutator of two free random variables $a$ and $b$ is $[a,b]=i\left(ab-ba\right)$, see \ts{Nica--Speicher} \cite{NicaSpeicher98}. The free random variable $[a,b]$ is bounded and self-adjoint, provided $a$ and $b$ are. In \cite{FuncBB} implicit equations for the maximum of the support of $[a,b]$ are obtained when $a$ and $b$ are two free standard semicircular random variables.

\ \\ \ \\


\ \\\ \\ \noindent \textsc{Mathematics Institute, University of Warwick, Coventry CV4 7AL, UK}\\\ \\ \textit{Email address} \texttt{j.ortmann@warwick.ac.uk}

\end{document}